\documentclass[11pt,reqno]{amsart}
\usepackage{amssymb,latexsym, amsmath, amsthm}
\usepackage{color}
\usepackage{graphicx}
\usepackage[mathcal]{euscript}
\usepackage{mathrsfs}
\usepackage{enumerate}
\usepackage{hyperref}

\newtheorem{theorem}{Theorem}[section]
\newtheorem{lemma}[theorem]{Lemma}
\newtheorem{proposition}[theorem]{Proposition}

\theoremstyle{definition}

\numberwithin{equation}{section}

\allowdisplaybreaks
\begin{document}

\title[Hankel operators on the Fock-Sobolev spaces]{Hankel operators on the Fock-Sobolev spaces}

\author[\textbf{Anuradha Gupta} and \textbf{Bhawna Gupta}]
{\textbf{Anuradha gupta} \and \textbf{Bhawna gupta}\\}

\address{Anuradha gupta, 
Delhi college of arts and commerce,
Department of mathematics,
University of Delhi,
Delhi,
India.}
\email{dishna2@yahoo.in}
\address{Bhawna Gupta,
Department of mathematics,
University of Delhi,
Delhi,
India.}
\email{swastik.bhawna26@gmail.com}

\begin{abstract} In this paper, we study operator-theoretic properties (boundedness and compactness) of Hankel operators on the Fock-sobolev spaces $ \mathscr{F}^{p,m} $ in terms of $ \mathcal{BMO}_r^p $ and $ \mathcal{VMO}_r^p $ spaces, respectively, for a non-negative integers $ m $, $ 1 \leq p < \infty $ and $ r > 0 $. Along the way, we also study Berezin transform of Hankel operators on $ \mathscr{F}^{p,m} $. The results in this article are analogous to Zhu's characterization and Per\"al\"a's characterization of bounded and compact Hankel operators on the Bergman spaces of unit disc and the weighted Fock spaces, respectively.
\end{abstract}



\maketitle

\section{Introduction}

 The investigation of Hankel operators on several spaces like Hardy spaces, Bergman spaces, Bergman spaces on a cartain domains, Fock spaces, Fock-type spaces etc., has a long history in mathematics and mathematical physics. We refer \cite{georghankel, karelhankel, zhubmo, zhuoperator, zhuanalysis} for the study on these spaces. Zhu \cite{zhubmo} obtained a characterization of bounded and compact Hankel operators on the Bergman space by defining the spaces of bounded mean oscillation and vanishing mean oscillation with respect to the Bergman metric and analogous results were obtained by Perala, Schuster and Virtanen \cite{peralahankel} on the weighted Fock spaces. After that Cho and Zhu \cite{chofock1, chofock} gave a very useful Fourier characterization of a new class of spaces known as Fock-Sobolev spaces on $ \mathbb{C}^n $ ($ n \geq 1 $). Our aim is to study the basic properties of Hankel operators on the Fock-Sobolev spaces. In particular, we examine the boundedness and compactness of these operators in terms of $ \mathcal{BMO}_r^p $ and $ \mathcal{VMO}_r^p $ spaces for the generating symbols.

 Let $ dA $ be the Lebesgue area measure on $ \mathbb{C} $. For $ 1 \leq p \leq \infty $, let $ \mathscr{F}^p $ be the space of all entire functions $ f $ on the complex plane $ \mathbb{C} $ such that $ g(v)  e^{- \frac{1}{2} |v|^2} \in L^p (\mathbb{C}, dA(v)) $ with norm  
\begin{equation*}
\| g \|_p = \left \{ \frac{p}{2\pi} \int_{\mathbb{C}} |g(v)|^p e^{- \frac{p}{2} |v|^2} dA(v)  \right\}^{\frac{1}{p}},
\end{equation*} for $ 1 \leq p < \infty $ and  
\begin{equation*} 
\| g \|_\infty = \underset{v \in \mathbb{C}}{\text{ ess sup }} \{ |g(v)| e^{- \frac{1}{2} |v|^2} \},
\end{equation*} for $  p = \infty. $

Throughout the paper, $ m $ is a fixed non-negative integer. The Fock-Sobolev space $ \mathscr{F}^{p,m} $ is the space of all entire functions $ g $ on $ \mathbb{C} $ such that 
\begin{equation*}
 \| g\|_{p,m} = \sum_{0 \leq k \leq m} \| f^{(k)} \|_p < \infty.
\end{equation*}

Cho and Zhu \cite{chofock} proved that $ g \in \mathscr{F}^{p,m} $ if and only if $ z^m g \in \mathscr{F}^p $ and $ \| g \|_{p,m} $ can be taken as 
\begin{equation*} 
\| g \|_{p,m} = \left \{ \omega_{p,m}  \int_{\mathbb{C}} | g(v)|^p |v|^{mp} e^{- \frac{p}{2} |v|^2} dA(v) \right \}^\frac{1}{p} ; 1 \leq p < \infty 
\end{equation*} and  
\begin{equation*} 
\| g \|_\infty = \underset{v \in \mathbb{C}}{\text{  sup }} { |g(v) v^m e^{- \frac{1}{2} |v|^2} |}  ; p = \infty,
\end{equation*}
where  $ \omega_{p,m} = (\frac{p}{2})^{\frac{mp}{2}+1} \frac{1}{\pi \Gamma (\frac{mp}{2}+1)} $.

Denote by $ L^{p,m} $, the space of all Lebesgue measurable functions $ g $ such that $ g(v) |v|^m e^{- \frac{1}{2} |v|^2} \in L^p(\mathbb{C}, dA(v))$ and $ \mathscr{F}^{p,m} $ is  a closed subspace of $ L^{p,m} $.

The space $ \mathscr{F}^{2,m} $, the closed subspace of the Hilbert space $ L^{2,m} $, is a Hilbert space with inner product 
\[ \left \langle f, g \right \rangle_{p,m} = \frac{1}{\pi} \int_\mathbb{C} f(v) \overline{g(v)} |v|^{2m} e^{- |v|^2} dA(v) \text{ for all } f, g \in   \mathscr{F}^{2,m},
\]and  having reproducing kernel 
\[ K^m(v,z)=  K_z^m(v) = \sum_{k=0}^\infty \frac{m!}{(k+m)!}( \bar{z} v)^k = m! \frac{( e^{\bar{z} v} - Q_m(\bar{z} v))}{(\bar{z} v)^m },
\]
where $ Q_m(w) $ is the Taylor polynomial of $ e^w $ of order $ (m-1) $ that is, $ Q_m(w) = \displaystyle \sum_{k = 0 }^{ m-1} w^k $. Let \[ k_z^m(v)  = \frac{K_z^m(v)}{\sqrt{K_z^m(z)}} = \frac{ ( e^{\bar{z}v} - Q_m(\bar{z}v))}{(\bar{z}v)^m} \left \{\frac{m!|z|^{2m}}{(e^{|z|^2} - Q_m (|z|^2))}\right \}^\frac{1}{2}\] denote the normalized reproducing kernel of $ \mathscr{F}^{2,m} $. Also, the sequence $ { b_k(v) }_{k=0}^\infty $ forms an orthonormal basis of $ \mathscr{F}^{2,m} $, where 
\[ 
b_k(v) = \sqrt{\frac{m!}{(k+m)!}} v^k .
\]
Cho and Zhu \cite{chofock} showed that the orthogonal projection $ P^m : L^{2,m} \rightarrow \mathscr{F}^{2,m} $ given by 
\[ P^m g(z)= \left \langle g, K_z^m \right \rangle_{2,m}  = \frac{1}{\pi} \int_\mathbb{C} g(v) \overline{K_z^m(v)} |v|^{2m} e^{- |v|^2} dA(v) 
\] is a bounded projection from $ L^{p,m} $ onto $ \mathscr{F}^{p,m} $ for $ 1 \leq p \leq \infty $.

\section{\bf $\mathcal{BMO}_r^p $ spaces and boundedness of Hankel operators on $ \mathscr{F}^{p,m}$}
Let $ \Omega_m $ denote the space of all Lebesgue measurable functions $ g $ on $ \mathbb{C} $ such that $ g k_v^m  \in \bigcup_{p' \geq 1}L^{p',m} $ for each $ v \in  \mathbb{C} $. Let $ I $ denotes the identity operator on $ L^{p,m} $.

The following result can be found in \cite{chofock1}, following which it is clear that for each $ v \in \mathbb{C} $, the reproducing kernel $ \|K_v^m\|_{p', m} $ is finite for all possible $ p' \geq 1 $.
\begin{lemma}  Suppose $ m $ is a fixed non-negative integers and $Q_m(z) $ is the Taylor polynomial of $ e^z $ of order $ m-1 $(with the convention that $Q_0 = 0$). For any parameter $ p'> 0 $, $ \sigma > 0 $, $c> 0 $ and $d > -mp'-2$, we can find a positive constant $ C_0 $ such that 
\[ \int_{\mathbb{C}} |e^{\bar{z}w} - Q_m{\bar{z}w} |^{p'} e^{-c|w|^2} |w|^{d}dA(w) \leq C_0 |z|^d e^{\frac{{p'}^2}{4c}|z|^2},
\] for all $ |z| \geq \sigma$. Furthermore, this holds for all $z$ if $ d \leq p'm $ as well.
\end{lemma}

Therefore, it follows that if $ g \in \Omega_m $, then the Hankel operator $ H_g^p : \mathscr{F}^{p,m} \rightarrow L^{p,m} $  with symbol $ g $, defined by $ H_g^p f = (I - P^m) gf $ for all $ f \in \mathscr{F}^{p,m} $, is densely defined on $ \mathscr{F}^{p,m} $, since the set of linear span of all kernel functions $ \{ k_v^m: v\in \mathbb{C} \} $ is dense in the space $ \mathscr{F}^{p,m} $. 
 By using the definition of $ P^m $, 
 we write 
 \begin{equation}\label{seq}
  H_g^p f(z) = \frac{1}{\pi} \int_\mathbb{C} (g(z) - g(v) ) f(v) \overline{K_z^m(v)} |v|^{2m} e^{- |v|^2} dA(v). 
\end{equation}
 Henceforth, for the convergence of integral in \eqref{seq}, we will assume that the symbol $ g  $ is in $ \Omega_m $.

For some $ z \in \mathbb{C} $, $ 1 \leq p < \infty $ and $ 0 < r < \infty $, let $ \mathscr{B}(z; r) = \{ v \in \mathbb{C}: | v - z | \leq r \} $ be the Euclidean disk centred at $ z $ and of radius $r$. Denote by $ L_{Loc}^p $, the space of all Lebesgue measurable functions $ g $ on $ \mathbb{C} $ such that $ g(v) \in L^p(K, dA(v)) $ for each compact subset $ K $ of $ \mathbb{C} $. Let $\mathcal{BA}_r $ be the set of all $ L_{Loc}^1$ integrable functions $ g $ on $ \mathbb{C} $ such that $ \tilde{g}_r $ defined by 
\[ \tilde{g}_r (z) = \frac{1}{ \pi r^2} \int_{\mathscr{B}(z; r)} g(v) dA(v) \] is bounded on $ \mathbb{C}$. For finite $ p \geq 1$ and $ g \in L_{Loc}^p $, denote  
\[ \tilde{g}_r ^p (z) = \frac{1}{ \pi r^2} \int_{\mathscr{B}(z; r)} |g(v)|^p  dA(v).\]
Let $ \mathcal{BA}_r^p $ be the set of all $ L_{Loc}^p $  integrable functions $ g $ on $ \mathbb{C} $ such that $\tilde{g}_r ^p $ is bounded on $ \mathbb{C}$.   
Let $ \mathcal{BMO}_r^p $ denote the set of all $ L_{Loc}^p $  integrable functions $ g $ such that 
\[ \| g \|_{\mathcal{BMO}_r^p} = \underset{z \in \mathbb{C}}{Sup} \left \{ \frac{1}{ \pi r^2} \int_{\mathscr{B}(z; r)} |g(v) - \tilde{g}_r  (z)|^p  dA(v)\right \}^{\frac{1}{p}} \] is finite.
Let $ \mathcal{BO}_r $ be the set of all continuous functions $ g $ on $ \mathbb{C} $ such that 
\[ \| g \|_{\mathcal{BO}_r} = \underset{z \in \mathbb{C}}{Sup} \left \{ \underset{v \in \mathscr{B}(z; r)}{Sup} | g(v) - g(z) |\right \} < \infty.
\]

The following lemmas will be instrumental in further study of Hankel operators on $ \mathscr{F}^{p,m} $.

\begin{lemma}  \cite{peralahankel} \label{lem1} Let $ p \geq 1 $. Then the following conditions hold:
\begin{enumerate}
\item Let $ g \in L_{Loc}^p $ then $ g \in \mathcal{BMO}_r^p $ if and only if there is a constant $ C > 0 $ such that for every $ z \in \mathbb{C} $ there exists a constant $ \mu_z $ such that 
\begin{equation*}
 \int_{\mathscr{B}(z; r)} |g(v) - \mu_z |^p  dA(v) \leq C.
\end{equation*} 
\item For $ 0 < r < R $, $ \mathcal{BMO}_R^p \subset \mathcal{BMO}_r^p $.
\item $ \mathcal{BO}_r $ is independent of $r$. Moreover, for any continous function $ g $ on $ \mathbb{C} $, $ g \in \mathcal{BO} $ if and only if there exists a constant $ C_0 > 0 $ such that 
 \begin{equation*}
 | g(z) - g(v) | \leq C_0(|z-v| +1)
\end{equation*}  for all $ z, v \in \mathbb{C}$. 
\item If $ g \in \mathcal{BMO}_{2r}^p $, then $ \tilde{g}_r \in \mathcal{BO}_r $.
\item If $ g \in \mathcal{BMO}_{2r}^p $, then $ g - \tilde{g}_r \in \mathcal{BA}_r^p $.
\item $ \mathcal{BMO}_r^p \subset \mathcal{BO}_r + \mathcal{BA}_r^p$ for $ 0 < r < \infty $.
\end{enumerate}
\end{lemma}

\begin{lemma} \label{lem2} \cite{chofock1}  Suppose $ t \in \mathbb{R} $ and $ M > 0 $ be a fixed real number.
\begin{enumerate} \item Then there exists a constant $ C_0 > 0 $  such that 
\begin{equation*} 
\sum_{k=0}^\infty \left(\frac{y}{k+1}\right)^t \frac{y^k}{k!} \leq C_0 e^y
\end{equation*} for all real $ y \geq M $. Futhermore, this holds for all $ y \geq 0 $ if $ t \geq 0 $.
\item Then there exists a constant $ C_0 > 0 $  such that 
\begin{equation*} 
\sum_{k=0}^\infty \left(\frac{y}{k+1}\right)^t \frac{y^k}{k!} \geq C_0 e^y
\end{equation*} for all real $ y \geq M $. Futhermore, this holds for all $ y \geq 0 $ if $ t \leq 0 $.
\end{enumerate}
\end{lemma}
For any two points $ u $ and $v$ such that $ u $ and $ v $ do not lie on the same ray emanating from the origin, the lattice generated by $ u $ and $ v $ is the set $ \{ au+bv| a,b \in \mathbb{Z} \}$. 

\begin{lemma} \label{lem3} \cite{wangtoeplitz}
Suppose $ \lambda $ is a locally integrable positive measure, $ p >0 $, $ r > 0 $,  $ m $ is a non- negative integer and $ \{ b_n \} $ is the lattice in $ \mathbb{C} $ generated by $ r $ and $ r i$. Then the following conditions are equivalent.
\begin{enumerate}
\item There exists a constant $ C_0 $ such that 
\begin{equation*}
\int_{\mathbb{C}} |g(v) v^m e^{- \frac{|v|^2}{2}}|^p d \lambda \leq C_0 \| g \|_{p,m}^p 
\end{equation*} for all entire functions $ g $.
\item There exists a constant $ C_0 > 0$ such that $ \lambda(\mathfrak{B}(z;r)) < C_0 $ for all $ z \in \mathbb{C} $.
\item There exists a constant $ C_0 > 0 $ such that $ \lambda(\mathfrak{B}(b_n ;r)) < C_0 $ for all positive integers $ n $.
\end{enumerate}
\end{lemma}

The Berezin transform of a function $ g $ is given by 
\begin{align*}
\mathfrak{B_m} (g)(z) &= \left \langle g k_z^m, k_z^m \right \rangle_{2, m} = \frac{1}{\pi m!} \int_{\mathbb{C}} g(v) |k_z^m(v) |^2 |v|^{2m} e^{-|v|^2} dA(v)\\
& = \frac{1}{\pi (e^{|z|^2} - Q_m (|z|^2)} \int_{\mathbb{C}} g(v) | e^{\bar{z}v} - Q_m(\bar{z}v)|^2 e^{-|v|^2} dA(v),
\end{align*}where $ k_z^m $ denotes the normalized reproducing kernel of $ \mathscr{F}^{2,m} $.

\begin{proposition}
\label{lem4} 
Let $ g \in \Omega_m $. For $ 1 \leq p \leq \infty $, the following conditions are equivalent:
\begin{enumerate}
\item $ g \in \mathcal{BA}_r^p$;
\item There exists a positive constant $ C $ such that \begin{equation*}
\frac{1}{\pi (e^{|z|^2} - Q_m (|z|^2)} \int_{\mathbb{C}}| g(v)|^p | e^{\bar{z}v} - Q_m(\bar{z}v)|^2 e^{-|v|^2} dA(v) \leq C
\end{equation*} for all $ z \in \mathbb{C}$;
\item  The multiplication operator $ L_g: \mathscr{F}^{p,m} \rightarrow L^{p,m} $ is bounded.  
\end{enumerate}
\end{proposition}

\begin{proof} $ (1) \Leftrightarrow (2) $
Let $ g \in \mathcal{BA}_r^p $ then $ \int_{\mathscr{B}(z; r)} |g(v)|^p  dA(v) $ is bounded on $ \mathbb{C} $. Then lemma \ref{lem3} gives 
\begin{equation*}
\int_{\mathbb{C}} |h(v) v^m e^{- \frac{|v|^2}{2}}|^p d \lambda \leq C_0 \| h \|_{p,m}^p 
\end{equation*} for all entire functions $ h $ where $ d\lambda(v) = |g(v)|^p dA(v) $  if and only if $ g \in \mathcal{BA}_r^p $ and hence, it follows that  $ g \in \mathcal{BA}_r^p $ if and only if $ \mathfrak{B_m}|g|^p $ is bounded on $ \mathbb{C}$ where 
\begin{equation*}
\mathfrak{B_m}| g(z)|^p = \frac{1}{\pi (e^{|z|^2} - Q_m (|z|^2)} \int_{\mathbb{C}}| g(v)|^p | e^{\bar{z}v} - Q_m(\bar{z}v)|^2 e^{-|v|^2} dA(v).
\end{equation*} 
$(1) \Leftrightarrow (3) $ Let $ g \in \mathcal{BA}_r^p $ then by definition $ \tilde{g}_r^p $ is bounded. Define  a non-negative measure $ d \lambda (z) = |g(z)|^p  dA(z) $ on $ \mathbb{C} $ then $ \lambda(\mathscr{B}(z;r)) = \int_{\mathscr{B}(z;r)} d\lambda(v) = \int_{\mathscr{B}(z;r)} |g(v)|^p  dA(v) $. Therefore, from lemma \ref{lem3}, it follows that \begin{equation*}
\int_{\mathbb{C}} |h(v) v^m e^{- \frac{|v|^2}{2}}|^p d \lambda \leq C_0 \| h \|_{p,m}^p 
\end{equation*} for all entire functions $ h $ if and only if  $ g \in \mathcal{BA}_r^p $. Thus, for all $ h \in \mathscr{F}^{p,m} $, we have
\begin{align*}
\| L_g(h)\|_{p,m}^p &= \| hg\|_{p,m}^p = \omega_{p,m} \int_{\mathbb{C}} |h(v)|^p |g(v)|^p |v|^{mp} e^{-\frac{p}{2} |v|^2 } dA(v)\\
& = \omega_{p,m} \int_{\mathbb{C}} |h(v)v^m e^{-\frac{1}{2} |v|^2 }|^p d\lambda (v) \leq C \| h \|_{p,m}^p 
\end{align*} for some constant $ C > 0 $.
\end{proof}

Thus, from lemmas \ref{lem1} and \ref{lem4}, it can be observed that $ \mathcal{BO}_r $ and $ \mathcal{BA}_r^p $ are independent of $ r $ and hence, we will denote them by $ \mathcal{BO} $ and $ \mathcal{BA}^p $, respectively.

\begin{lemma} \label{lem5}
Let $ g \in \mathcal{BMO}_r^p $. Then 
\begin{equation*}
\frac{1}{\pi (e^{|z|^2} - Q_m (|z|^2)} \int_{\mathbb{C}} |g(v)- \mathfrak{B_m}g(z)|^p | e^{\bar{z}v} - Q_m(\bar{z}v)|^2 e^{-|v|^2} dA(v) 
\end{equation*} is bounded for  $ |z| > M $, for some positive constant $ M $.
\end{lemma}

\begin{proof}
Let $ g \in \mathcal{BMO}_r^p \subset \mathcal{BO} + \mathcal{BA}_r^p $, therefore, there exist two functions $ g_+ , g_- $ on $\mathbb{C} $ such that $ g_+ \in \mathcal{BO}_r $ and $ g_- \in \mathcal{BA}_r^p $.
Since $ g_+ \in \mathcal{BO}_r $, therefore, by using lemma \ref{lem2} and the fact that $ \mathcal{BO}_r $ is independent of $ r $ and 
\begin{equation*}
 \lim_{\substack{v \in \mathscr{B}(z; r)\\ |z| \rightarrow \infty}} ( 1 - e^{-\bar{z}v}Q_m(\bar{z}v)) = 1,
\end{equation*}
 it follows that 
\begin{align*}
& \left \{ \frac{1}{\pi (e^{|z|^2} - Q_m (|z|^2)} \int_{\mathscr{B}(z;r)} |g_+(v)- \mathfrak{B_m}g_+(z)|^p | e^{\bar{z}v} - Q_m(\bar{z}v)|^2 e^{-|v|^2} dA(v) \right \}\\
=& \left \{ \frac{1}{\pi (e^{|z|^2} - Q_m (|z|^2)} \int_{\mathscr{B}(z;r)} |g_+(v)- \mathfrak{B_m}g_+(z)|^p |e^{\bar{z}v}|^2 | 1 - e^{-\bar{z}v}Q_m(\bar{z}v)|^2 e^{-|v|^2} dA(v)  \right \}\\
& \leq C \left \{  \int_{\mathscr{B}(z;r)} |g_+(v)- \mathfrak{B_m}g_+(z)|^p e^{-|z-v|^2} dA(v) \right \}\\
& \leq C \left \{  \int_{\mathbb{C}} |g_+(v)- \mathfrak{B_m}g_+(z)|^p e^{-|z-v|^2} dA(v) \right \}\\
& = C \left \{  \int_{\mathbb{C}} |g_+(z-v)- \mathfrak{B_m}g_+(z)|^p e^{-|v|^2} dA(v) \right \}, 
\end{align*} for all $|z| > M $ for some positive constants $C $ and $M $, where 
\begin{align*}
&| g_+(z-v)- \mathfrak{B_m}g_+(z)|\\
& = |g_+(z-v)- \frac{1}{\pi (e^{|z|^2} - Q_m (|z|^2)} \int_{\mathbb{C}} g_+(u) | e^{\bar{z}u} - Q_m(\bar{z}u)|^2 e^{-|v|^2} dA(u)|\\
& = |\frac{1}{\pi (e^{|z|^2} - Q_m (|z|^2)} \int_{\mathbb{C}} (g_+(z-v)- g_+(u)) | e^{\bar{z}u} - Q_m(\bar{z}u)|^2 e^{-|v|^2} dA(u)|\\
& = | \displaystyle \lim_{r \rightarrow \infty } \frac{1}{\pi (e^{|z|^2} - Q_m (|z|^2)} \int_{\mathscr{B}(z;r)} ( g_+(z-v)- g_+(u)) | e^{\bar{z}u} - Q_m(\bar{z}u)|^2 e^{-|u|^2} dA(u)|\\
& \leq  \displaystyle  \lim_{r \rightarrow \infty } |\frac{1}{\pi (e^{|z|^2} - Q_m (|z|^2)} \int_{\mathscr{B}(z;r)}( g_+(z-v)- g_+(u)) | e^{\bar{z}u} |^2 | 1 - e^{-\bar{z}u} Q_m(\bar{z}u)|^2 e^{-|u|^2} dA(u)|\\
& \leq C \displaystyle  \lim_{r \rightarrow \infty }  |\int_{\mathscr{B}(z;r)}( g_+(z-v)- g_+(u)) |  e^{-|z -u|^2} dA(u)|\\
& \leq C | \int_{\mathbb{C}} ( g_+(z-v)-g_+(u))   e^{-|z-u|^2} dA(u)|\\
& = C  \int_{\mathbb{C}}| g_+(z-v)-g_+(z - u)|   e^{-|u|^2} dA(u)|
\end{align*} for all $|z| > M $.
Therefore, 
\begin{align*}
& \left \{ \frac{1}{\pi (e^{|z|^2} - Q_m (|z|^2)} \int_{\mathscr{B}(z;r)} |g_+(v)- \mathfrak{B_m}g_+(z)|^p | e^{\bar{z}v} - Q_m(\bar{z}v)|^2 e^{-|v|^2} dA(v) \right \}\\
& \leq C^2 \iint_{\mathbb{C}} | g_+(z-v)-g_+(z - u)|^p   e^{-|u|^2} dA(u)e^{-|v|^2} dA(v)\\
& \leq \iint_{\mathbb{C}}(|u-v| +1)^p  e^{-|u|^2} dA(u)e^{-|v|^2} dA(v)\\
\end{align*} for all $|z| > M $.
Now, since $ g_- \in \mathcal{BA}_r^p $, therefore by lemma \ref{lem4}, there exists a positive constant $ C $ such that 
\begin{equation*}
\frac{1}{\pi (e^{|z|^2} - Q_m (|z|^2)} \int_{\mathbb{C}}| g_-(v)|^p | e^{\bar{z}v} - Q_m(\bar{z}v)|^2 e^{-|v|^2} dA(v) \leq C
\end{equation*} 
for all $ z \in \mathbb{C}$. Consider 

\begin{align*}
& \left \{ \frac{1}{\pi (e^{|z|^2} - Q_m (|z|^2)} \int_{\mathbb{C}} |g_-(v)- \mathfrak{B_m}g_-(z)|^p | e^{\bar{z}v} - Q_m(\bar{z}v)|^2 e^{-|v|^2} dA(v) \right \}^\frac{1}{p} \\
& \leq \left \{ \frac{1}{\pi (e^{|z|^2} - Q_m (|z|^2)} \int_{\mathbb{C}} |g_-(v)|^p | e^{\bar{z}v} - Q_m(\bar{z}v)|^2 e^{-|v|^2} dA(v) \right \}^\frac{1}{p} + | \mathfrak{B_m}g_-(z)| \| k_z \|_{2,m}^\frac{2}{p}\\
& \leq \left \{ \frac{1}{\pi (e^{|z|^2} - Q_m (|z|^2)} \int_{\mathbb{C}} |g_-(v)|^p | e^{\bar{z}v} - Q_m(\bar{z}v)|^2 e^{-|v|^2} dA(v) \right \}^\frac{1}{p} + | \mathfrak{B_m}g_-(z)|\\
& \leq C + | \mathfrak{B_m}g_-(z)|,
\end{align*}
where 
\begin{align*}
| \mathfrak{B_m}g_-(z)| & \leq  \frac{1}{\pi (e^{|z|^2} - Q_m (|z|^2)} \int_{\mathbb{C}} |g_-(v)| | e^{\bar{z}v} - Q_m(\bar{z}v)|^2 e^{-|v|^2} dA(v) \\
& \leq \left \{\frac{1}{\pi (e^{|z|^2} - Q_m (|z|^2)} \int_{\mathbb{C}} |g_-(v)|^p | e^{\bar{z}v} - Q_m(\bar{z}v)|^2 e^{-|v|^2} dA(v) \right \}^\frac{1}{p} \| k_z \|_{2,m}^\frac{2}{q}\\
& = \left \{\frac{1}{\pi (e^{|z|^2} - Q_m (|z|^2)} \int_{\mathbb{C}} |g_-(v)|^p | e^{\bar{z}v} - Q_m(\bar{z}v)|^2 e^{-|v|^2} dA(v) \right \}^\frac{1}{p}\\
& \leq C.
\end{align*}
Therefore, 
\begin{equation*}
 \left \{ \frac{1}{\pi (e^{|z|^2} - Q_m (|z|^2)} \int_{\mathbb{C}} |g_-(v)- \mathfrak{B_m}h(z)|^p | e^{\bar{z}v} - Q_m(\bar{z}v)|^2 e^{-|v|^2} dA(v) \right \}^\frac{1}{p} \leq 2 C 
\end{equation*}  and hence, we get the result.
\end{proof}

\begin{lemma} \label{lem6}
Suppose there exists a positive constant $ M $ such that
\begin{equation*}
\underset{|z |> M}{Sup} \left \{ \frac{1}{\pi (e^{|z|^2} - Q_m (|z|^2)} \int_{\mathbb{C}} |g(v)- \mathfrak{B_m}g(z)|^p | e^{\bar{z}v} - Q_m(\bar{z}v)|^2 e^{-|v|^2} dA(v) \right \}, 
\end{equation*} is bounded. Then, there exists a constant $ M' > 0 $ such that for each $ z \in \mathbb{C} $, there exists a constant $ \mu_z $ such that 
\begin{equation*}
\underset{|z| > M'}{Sup} \left \{ \frac{1}{ \pi r^2} \int_{\mathscr{B}(z; r)} |g(v) - \mu_z|^p  dA(v)\right \} 
\end{equation*} is bounded.
\end{lemma}

\begin{proof} By using the fact that $ e^{-|z-v|^2} \geq a  $ for $ v \in \mathscr{B}(z;r) $ and for some constant $ a > 0 $, it follows that  
\begin{align}
&a C \left \{ \frac{1}{ \pi r^2} \int_{\mathscr{B}(z; r)} |g(v) - \mathfrak{B_m}g(z)|^p  dA(v)\right \}\nonumber \\
& \leq C \left \{ \frac{1}{ \pi r^2} \int_{\mathscr{B}(z; r)} |g(v) - \mathfrak{B_m}g(z)|^p e^{-|z-v|^2} dA(v)\right \}\nonumber \\
& = C \left \{ \frac{1}{ \pi r^2} \int_{\mathscr{B}(z; r)} |g(v) - \mathfrak{B_m}g(z)|^p  e^{-|z|^2} |e^{\bar{z}v}|^2 e^{-|v|^2}dA(v)\right \}\nonumber\\
\label{eq1} &\leq \left \{ \frac{1}{ \pi^2 r^2((e^{|z|^2} - Q_m (|z|^2))} \int_{\mathscr{B}(z; r)} |g(v) - \mathfrak{B_m}g(z)|^p | e^{\bar{z}v} - Q_m(\bar{z}v)|^2 e^{-|v|^2}dA(v)\right \}\\
& \leq \left \{ \frac{1}{ \pi^2 r^2((e^{|z|^2} - Q_m (|z|^2))} \int_{\mathbb{C}} |g(v) - \mathfrak{B_m}g(z)|^p | e^{\bar{z}v} - Q_m(\bar{z}v)|^2 e^{-|v|^2}dA(v)\right \}\nonumber
\end{align} for all $ |z| > M' $, where the Eq. \eqref{eq1} follows from lemma \ref{lem2} and \begin{equation*}
 \lim_{\substack{v \in \mathscr{B}(z; r)\\ |z| \rightarrow \infty}} ( 1 - e^{-\bar{z}v}Q_m(\bar{z}v)) = 1. \qedhere
\end{equation*}
\end{proof}

Lemmas \ref{lem6} and \ref{lem7} jointly lead to the following result:
\begin{theorem}\label{thm1}
Let $ g \in \mathcal{BMO}_r^p$. Then the following conditions are equivalent:
\begin{enumerate}
\item There exists a constant $ M > 0 $ such that \begin{equation*}
\underset{|z| > M}{Sup} \left\{ \frac{1}{\pi (e^{|z|^2} - Q_m (|z|^2)} \int_{\mathbb{C}} |g(v)- \mathfrak{B_m}g(z)|^p | e^{\bar{z}v} - Q_m(\bar{z}v)|^2 e^{-|v|^2} dA(v)\right \} < \infty;
\end{equation*} 
\item There exists a constant $ M > 0 $ such that for each $ z \in \mathbb{C}  $, there exists a constant $ \mu_z $ such that 
\begin{equation*}
\underset{|z| > M}{Sup} \left \{ \frac{1}{ \pi r^2} \int_{\mathscr{B}(z; r)} |g(v) - \mu_z|^p  dA(v)\right \} < \infty;
\end{equation*} 
\item There exists a constant $ M > 0 $ such that for each $ z \in \mathbb{C}  $, there exists a constant $ \mu_z $ such that 
 \begin{equation*}
\underset{|z| > M}{Sup}\left \{ \frac{1}{\pi (e^{|z|^2} - Q_m (|z|^2)} \int_{\mathbb{C}} |g(v)- \mu_z |^p | e^{\bar{z}v} - Q_m(\bar{z}v)|^2 e^{-|v|^2} dA(v)\right\} < \infty.
\end{equation*} 
\end{enumerate}
\end{theorem}

\begin{proof} $ (1) \Leftrightarrow (3) $ This follows on the proof of lemma $ 3.1 $ of \cite{peralahankel}.
\end{proof}

\begin{proposition}\label{lem7}  Let $ g \in \mathcal{BMO}_{2r}^p $. Then there exists a positive constant $ M $ such that the followings hold:

(1)  
$
\underset{|z |>M +r}{Sup}\left\{ \underset{v \in \mathscr{B}(z; r)}{Sup} | \mathfrak{B_m}g(v) - \mathfrak{B_m}g(z) |\right\} < \infty 
$

(2)
$
\underset{|z |>M +r }{Sup} \left\{\frac{1}{ \pi r^2} \int_{\mathscr{B}(z; r)} |(g - \mathfrak{B_m}g)(v)|^p  dA(v) \right\}< \infty.
$ 
\end{proposition}

\begin{proof} 
Let $ g \in \mathcal{BMO}_{2r}^p \subset \mathcal{BMO}_{r}^p $.
Consider 
\begin{align*}
&| \mathfrak{B_m}g(z) - \tilde{g}_r(z) | \\
& = |\frac{1}{ \pi r^2} \int_{\mathscr{B}(z; r)} \mathfrak{B_m}g(z)  dA(v)- \frac{1}{ \pi r^2} \int_{\mathscr{B}(z; r)} g (v)  dA(v)|\\
& \leq  \frac{1}{ \pi r^2} \int_{\mathscr{B}(z; r)} |g(v)- \mathfrak{B_m}g(z)|  dA(v)\\
& \leq  \left \{ \frac{1}{ \pi r^2} \int_{\mathscr{B}(z; r)} |g(v)- \mathfrak{B_m}g(z)|^p  dA(v)\right\}^\frac{1}{p}\\
& \leq C \left \{ \frac{1}{\pi (e^{|z|^2} - Q_m (|z|^2)} \int_{\mathbb{C}} |g(v)- \mathfrak{B_m}g(z)|^p | e^{\bar{z}v} - Q_m(\bar{z}v)|^2 e^{-|v|^2} dA(v)\right \}^\frac{1}{p}\\
& \leq C_0,
\end{align*} for all $ |z| > M >0 $ and for some constant $C_0>0 $. Thus, $$ \underset{|z |>M +r}{Sup} \left\{\underset{v \in \mathscr{B}(z; r)}{Sup} |( \mathfrak{B_m}g -\tilde{g}_r) (v) - (\mathfrak{B_m} - \tilde{g}_r)g(z) | \right\} <\infty $$ and $$ \underset{|z |> M +r}{Sup} \left\{\frac{1}{ \pi r^2} \int_{\mathscr{B}(z; r)} |( \mathfrak{B_m}g - \tilde{g}_r)(v)|^p  dA(v)\right\}<\infty.
$$  Since $ g \in \mathcal{BMO}_{2r}^p \subset \mathcal{BMO}_{r}^p $, therefore by lemma \ref{lem1}, it follows that $ \tilde{g}_r \in \mathcal{BO} $ and $ g -  \tilde{g}_r \in \mathcal{BA}^p $, so  $ \mathfrak{B_m}g = \mathfrak{B_m}g - \tilde{g}_r + \tilde{g}_r $, $ g - \mathfrak{B_m}g = g - \tilde{g}_r + \tilde{g}_r - \mathfrak{B_m}g $, we get the desired result.
\end{proof}

\begin{lemma} \label{lem8}
If $ g \in \mathcal{BA}^p $, then $ H_g^p $ is bounded on $ \mathscr{F}^{p,m}$ for finite $ p \geq 1$. 
\end{lemma}

\begin{proof}
By using lemma \ref{lem4}, $ g \in \mathcal{BA}^p $ if and only if $ L_g $ is bounded on $ \mathscr{F}^{p,m} $ if and only if $ H_g $ is bounded on $ \mathscr{F}^{p,m}$.
\end{proof}

\begin{lemma} \label{lem9}
If $ g \in \mathcal{BO} $, then $ H_g^p $ is bounded on $ \mathscr{F}^{p,m} $ for all $ 1 \leq p \leq \infty$.
\end{lemma}

\begin{proof} 
Let $ h \in \mathscr{F}^{p,m} $ and $1 < p <\infty  $. Since $ g \in \mathcal{BO} $, therefore, by lemma \ref{lem1}, we obtain that 
\begin{align*}
|H_g(h)(z)|^p  &\leq \left \{\omega_{p,m} \int_{\mathbb{C}} |g(z) - g(v))||h(v)| |\frac{e^{z \bar{v}} - Q_m(z \bar{v})}{(z \bar{v})^m }| e^{-|v|^2} |v|^{2m} dA(v)\right \}^p\\
& \leq C \left \{ \omega_{p,m} \int_{\mathbb{C}} (|z-v|+1)|h(v)| |\frac{e^{z \bar{v}} - Q_m(z \bar{v})}{(z \bar{v})^m }| e^{-|v|^2} |v|^{2m} dA(v) \right \}^p.
\end{align*}
 By Using (see \cite{chofock})\[ |\frac{e^{z \bar{v}} - Q_m(z \bar{v})}{(z \bar{v})^m }| \lesssim \frac{e^{\frac{1}{2}|z|^2 + \frac{1}{2} |v|^2 - \frac{1}{8}|z-v|^2}}{(1 + |z||v|)^m} \leq \frac{e^{\frac{1}{2}|z|^2 + \frac{1}{2} |v|^2 - \frac{1}{8}|z-v|^2}}{ (|z||v|)^m},\]
  it follows that
\begin{align*}
&|H_g(h)(z)|^p e^{-\frac{p}{2}|z|^2} |z|^{pm} \\
& \leq C  e^{-\frac{p}{2}|z|^2} |z|^{pm} \left \{ \omega_{p,m} \int_{\mathbb{C}} (|z-v|+1)|h(v)| |\frac{e^{z \bar{v}} - Q_m(z \bar{v})}{(z \bar{v})^m }| e^{-|v|^2} |v|^{2m} dA(v) \right \}^p\\
& \leq C  \left \{ \omega_{p,m} \int_{\mathbb{C}} (|z-v|+1)|h(v)| e^{-\frac{1}{2}|v|^2} e^{- \frac{1}{8}|z-v|^2} |v|^{m} dA(v) \right \}^p\\
& \leq C  \left \{ \omega_{p,m} \int_{\mathbb{C}} |h(v)|^p e^{-\frac{p}{2}|v|^2} |v|^{pm} dA(v) \right \}  \left \{ \omega_{p,m} \int_{\mathbb{C}} (|z-v|+1)^q  e^{- \frac{q}{8}|z-v|^2} dA(v) \right \}^\frac{p}{q}\\
& = C \| h\|_{p,m}^p  \left \{ \omega_{p,m} \int_{\mathbb{C}} (|z-v|+1)^q  e^{- \frac{q}{8}|z-v|^2} dA(v) \right \}^\frac{p}{q},
\end{align*} for some constant $ C >0$. Therefore, $\| H_f(h) \|_{p,m}^p \leq C_0 \| h\|_{p,m}^p $ and hence, $ H_f $ is bounded on $ \mathscr{F}^{p,m} $ for $1 < p < \infty$.
For $ p=1 $, we can conclude by using Fubini's theorem  that 
\begin{align*}
 \| H_g(h)\|_{1,m} & = \omega_{1,m}\int_\mathbb{C} |H_g(h)(z)| e^{-\frac{1}{2}|z|^2} |z|^{m} dA(z)\\
 & \leq  C _0 \left \{  \int_{\mathbb{C}} |h(v)| e^{-\frac{1}{2}|v|^2} |v|^{m} dA(v) \right \}  \left \{  \int_{\mathbb{C}} (|z-v|+1)  e^{- \frac{1}{8}|z-v|^2} dA(z) \right \} \\
 & \leq C_0 \| h\|_{1,m},  
\end{align*}where $ C_0 $ is a constant. For $ p = \infty $, 
\begin{align*}
\| H_g(h)\|_{\infty ,m} & = |H_g(h)(z)| e^{-\frac{1}{2}|z|^2} |z|^{m}\\
& \leq \| h\|_{\infty,m} \left \{  \int_{\mathbb{C}} (|z-v|+1)  e^{- \frac{1}{8}|z-v|^2} dA(z) \right \} \leq C_1,
\end{align*} where $ C_1 >0 $ is a constant and hence, the result follows for all $ 1 \leq p \leq \infty $.
\end{proof}

\begin{theorem} \label{thm2} Let $ g \in \mathcal{BMO}_r^p $. Then the operators $ H_g^p $ and $ H_{\bar{g}}^p $ are bounded for all $ 1 \leq p < \infty $.
\end{theorem}
\begin{proof}
By using lemmas \ref{lem1}, \ref{lem8} and \ref{lem9} and the fact that if $ g \in \mathcal{BMO}_r^p $ then so is $ \bar{g} $, the result follows. 
\end{proof}

\section{\bf $\mathcal{VMO}_r^p $ spaces and Compactness of Hankel operators on $ \mathscr{F}^{p,m}$}
Define $ \mathcal{VA}_r $ be the set of all $ L_{Loc}^1 $  integrable functions $ g $ on $ \mathbb{C} $ such that $ \displaystyle \lim_{|z| \rightarrow \infty}\tilde{g}_r   = 0 $.
 For finite $ p \geq 1$, let $ \mathcal{VMO}_r^p $ denote the set of all $ L_{Loc}^p $  integrable functions $ g $ such that 
\[  \displaystyle \lim_{|z| \rightarrow \infty} \left \{ \frac{1}{ \pi r^2} \int_{\mathscr{B}(z; r)} |g(v) - \tilde{g}_r  (z)|^p  dA(v)\right \}^{\frac{1}{p}} =0. \] 
Let $ \mathcal{VO}_r \subset \mathcal{BO}_r $ be the set of all continuous functions $ g $ on $ \mathbb{C} $ such that 
\[  \displaystyle \lim_{|z| \rightarrow \infty} \underset{v \in \mathscr{B}(z; r)}{Sup} | g(v) - g(z) | =0.
\]
Let $ \mathcal{VA}_r^p $ be the set of all $ L_{Loc}^p $  integrable functions $ g $ on $ \mathbb{C} $ such that $ \displaystyle \lim_{|z| \rightarrow \infty}\tilde{g}_r^p   = 0 $.

The following lemma will be useful in the study of compact Hankel operators on $\mathscr{F}^{p,m} $ and the related results.

\begin{lemma} \label{lem10} \cite{wangtoeplitz} Let $ \lambda $ is a positive Borel measure, $ 0 < p < \infty $, $ r > 0 $, $ m $ is a non- negative integer and $ \{ b_n \} $ is the lattice in $ \mathbb{C} $ generated by $ r $ and $ ri$. Then the following conditions are equivalent.
\begin{enumerate}
\item $ 
 \displaystyle \lim_{n \rightarrow \infty} \int_{\mathbb{C}} |g_n (v) v^m e^{- \frac{|v|^2}{2}}|^p d \lambda =0  
$ for all bounded sequence $ \{g_n\} $ in $ \mathscr{F}^{p,m} $ that converges to $ 0 $ uniformly on compact sets;
\item  $ \displaystyle \lim_{|z| \rightarrow \infty} \lambda(\mathscr{B}(z;r)) = 0$;
\item $ \displaystyle \lim_{n \rightarrow \infty} \lambda(\mathscr{B}(b_n;r)) = 0$.
\end{enumerate}
\end{lemma}

Similar to $ \mathcal{BO}_r $ and $ \mathcal{BA}_r^p $, it is easy to observe that $ \mathcal{VO}_r $ and $ \mathcal{VA}_r^p $ are independent of $r$, so we will denote them by $\mathcal{VO} $ and $ \mathcal{VA}^p $, respectively.

The following results are analogues to theorem \ref{thm1} and lemma \ref{lem1}.

\begin{theorem} \label{thm3} Let $ p $ is any natural number. Then the following conditions are equivalent:
\begin{enumerate}
\item $ g \in \mathcal{VMO}^p $;
\item $ g \in \mathcal{VO} + \mathcal{VA}^p $;
\item  \begin{equation*}
\displaystyle \lim_{|z| \rightarrow \infty} \left \{ \frac{1}{\pi (e^{|z|^2} - Q_m (|z|^2)} \int_{\mathbb{C}} |g(v)- \mathfrak{B_m}g(z)|^p | e^{\bar{z}v} - Q_m(\bar{z}v)|^2 e^{-|v|^2} dA(v)\right \} =0;
\end{equation*} 
\item There exists a constant $ M > 0 $ such that for each $ z \in \mathbb{C}  $, there exists a constant $ \mu_z $ such that 
\begin{equation*}
\displaystyle \lim_{|z| \rightarrow \infty} \left \{ \frac{1}{ \pi r^2} \int_{\mathscr{B}(z; r)} |g(v) - \mu_z|^p  dA(v)\right \} =0;
\end{equation*} 
\item There exists a constant $ M > 0 $ such that for each $ z \in \mathbb{C}  $, there exists a constant $ \mu_z $ such that 
 \begin{equation*}
\displaystyle \lim_{|z| \rightarrow \infty} \left \{ \frac{1}{\pi (e^{|z|^2} - Q_m (|z|^2)} \int_{\mathbb{C}} |g(v)- \mu_z |^p | e^{\bar{z}v} - Q_m(\bar{z}v)|^2 e^{-|v|^2} dA(v) \right \} =0.
\end{equation*}
\end{enumerate}
\end{theorem}
 From theorem \ref{thm3}, it follows that $ \mathcal{VMO}_r^p $ is independent of $ r $, so we will write $ \mathcal{VMO}^p $.

\begin{lemma}  \label{lem11}  
\begin{enumerate} 
 \item If $ g \in \mathcal{VMO}^p $, then $ \tilde{g}_r \in \mathcal{VO} $.
\item If $ g \in \mathcal{VMO}^p $, then $ g - \tilde{g}_r \in \mathcal{VA}^p $ for every $ r > 0$.
\item If $ g \in \mathcal{VMO}^p $, then $\mathfrak{B_m}(g) \in \mathcal{VO}$.
\item If $ g  \in \mathcal{VMO}^p $, then $ g - \mathfrak{B_m}(g) \in \mathcal{VA}^p$.
\item  The function $ g \in \mathcal{VO} $ if and only if for each constant $ C > 0 $, there exists $ r > 0 $ such that $ | g(z) - g(v) | \leq C (1 + | z-v |) $ for all $ z, v \in \mathbb{C} \backslash \mathscr{B}(0;r)$ (see \cite{bauermean}).
\end{enumerate}
\end{lemma}

\begin{lemma}\cite{bauermean} \label{lem12} For $ r > 0 $, consider a function $ g: \mathbb{C} \backslash \mathscr{B}(0;r) \rightarrow \mathbb{C} $ with 
\[ | g(z) - g(v) | \leq C (1 + | z-v |) \text{ for all } z, v \in \mathbb{C} \backslash \mathscr{B}(0;r), 
\] where $ C > 0 $ is independent of $ g $. Then there exists  a function $ G $ on $ \mathbb{C} $ such that $ g = G $ on $ \mathbb{C} \backslash \mathscr{B}(0;r) $  and $ |G(z) - G(v) | \leq 2C (1 + | z-v |) $ for all $ z, v \in \mathbb{C} $.
\end{lemma}

\begin{theorem} \label{thm4} Let $ 1 \leq p <\infty $. Then $ g \in \mathcal{VMO}^p $ then the Hankel operators $ H_g^p $ and $ H_{\bar{g}}^p $ are both compact. 
\end{theorem}

\begin{proof} Let $ g \in \mathcal{VA}^p $. This gives the positive measure $ d \lambda = |g|^p d A $ satisfying $  \displaystyle \lim_{|z| \rightarrow \infty} \lambda(\mathscr{B}(z;r)) = 0$. So, by lemma \ref{lem10}, the multiplication operator $ L_g : \mathscr{F}^{p,m} \rightarrow L^{p,m} $ is compact and so is $ H_g^p $. 

Let $ g \in \mathcal{VO} $. Let $ \epsilon > 0 $ be arbitrary. Using lemma \ref{lem11} and \ref{lem12}, it follows that there exists a function $ G $ on $ \mathbb{C} $ such that $ g = G $ on $ \mathbb{C} \backslash \mathscr{B}(0;r) $  and $ |G(z) - G(v) | \leq 2\epsilon (1 + | z-v |) $ for all $ z, v \in \mathbb{C} $. Then lemma \ref{lem1} and \ref{lem9} give $ H_G^p $ is bounded with $\| H_G^p\| \leq 2 \epsilon C_0 $ for some constant $ C_0 > 0$. Also, $ H_{g - G}^p $ is compact, since $ g-G $ has compact support, and $ \| H_G^p - H_{g-G}^p \| = \| H_g^p \| \leq 2\epsilon C_0 $. Since $ \epsilon >0 $ is arbitrary, therefore, the Hankel operators $ H_g^p $ is compact. Similarly, it can be proved that the Hankel operators $ H_{\bar{g}}^p $ is compact.
Hence, by using theorem \ref{thm3}, the result follows.
\end{proof}

% ------------------------------------------------------------------------

\subsection*{Acknowledgment}
Support of UGC Research Grant to second author for carrying out the research work is gratefully acknowledged.

% ------------------------------------------------------------------------
\end{document}